\documentclass[10pt]{amsart}
\usepackage{amssymb}
\usepackage{amssymb}
\usepackage{amsmath}
\usepackage{amsfonts}

\usepackage{multicol, color}

\newtheorem{thm}{Theorem}
\newtheorem{rem}{Remark}
\newtheorem{lem}{Lemma}

\newtheorem{claim}{Claim}

\numberwithin{equation}{section}

\title[Energy Cascades in Forced NSE]{On energy cascades in non-homogeneous 3D Navier-Stokes Equations}

\author{R. Dascaliuc$^{1}$}\address{$^1$Department of Mathematics\\
Oregon State University\\ Corvallis, OR 97331} 
\author{Z. Gruji\'c$^2$}
\address{$^2$Department of Mathematics\\
University of Virginia\\ Charlottesville, VA 22904}
\email[R. Dascaliuc]{dascalir@math.oregonstate.edu}

\thanks{R.D.~was supported in part by National Science Foundation grant DMS-1211413.  Z.G.~acknowledges support of the Research Council of Norway
via grant 213474/F20 and the NSF via grant DMS-1212023.}


\begin{document}

\begin{abstract}
We show -- in the framework of physical scales and $(K_1,K_2)$-averages  -- that Kolmogorov's dissipation law combined with the smallness condition on a Taylor length scale are sufficient to guarantee energy cascades in the forced  Navier-Stokes equations. Moreover, in the periodic case we establish restrictive scaling laws -- in terms of Grashof number -- for kinetic energy, energy flux, and energy dissipation rate. These are used to improve our sufficient condition for forced cascades in physical scales. 
\end{abstract}

\maketitle


\section{Introduction}

This paper aims at better understanding of mathematical mechanisms behind  the phenomenon of energy cascade, one of the main features of turbulence theory dating back to Kolmogorov \cite{Kol41a,Kol41b}.
This feature, often observed in physical experiments, involves formation of  a wide range of length scales, {\em inertial range}, where viscous effects are dominated by the transport of energy to lower scales. This idea that can be rephrased as a statement of near-constancy of energy flux across the inertial range:
\begin{equation}\label{en-casc-law}\langle \Phi \rangle_{R}\sim \varepsilon,\end{equation}
where $\langle \Phi \rangle_{R}$ is a certain average of the energy flux $\Phi$ at the suitably defined length scale $R$, and $\varepsilon$ is the average energy dissipation rate. 

\medskip

The precise mathematical meaning of the average and of the length scale varies. In homogeneous, isotropic turbulence, it is convenient to use the scales provided by Fourier wave-numbers (See e.g. \cite{Fri95}). Moreover, spectral approach is often used to interpret the experimental data obtained by devices taking measurements on physical scales (\cite{Gol96}). Physical scales become important when looking at the geometric structures, e. g. vortices in the turbulent flow. The fact that the empirical laws like (\ref{en-casc-law}) are confirmed experimentally regardless of the measurement methodology involved, points to deeper ergodicity properties of fluid flows, yet to be established mathematically from the underlying equations of motion. 

\medskip

These observations raise a question of whether the phenomenon of energy cascade could be mirrored in the equations of motions, in particular, the 3D Navier-Stokes equations (NSE). Such question is also related to the open problem of regularity for the Navier-Stokes system. Indeed, the transfer of energy to lower scales may be connected to a mechanism for a possible finite-time blow-up of the solutions (\cite{KP05,Che08,CF09,Tao14}).

\medskip

The  momentum equation in the incompressible Navier-Stokes system can be written as 
\begin{equation}\label{intro_NSE}
\begin{aligned}
\partial_t u -\nu\Delta u + (u\cdot\nabla)  u + \nabla p = f\;,
\end{aligned}
\end{equation}
where $u(x,t)$ represents Eulerian velocity of the fluid, $p(x,t)$ is the pressure, and $f(x,t)$ is the external force. The linear term $\nu\Delta u$ represents dissipation, and therefore is used to define the energy dissipation rate $\varepsilon$ (see Section 2). The nonlinearity $(u\cdot\nabla)  u + \nabla p$ represents transport of kinetic and potential energy between the scales of the fluid. Consequently, this term is used to define the total energy flux $\Phi$. Thus, mathematically, the question of energy redistribution between the scales of the flow  is connected with a careful analysis of the relative magnitudes of these terms in appropriate settings. The presence of the forcing term $f$ further complicates the picture. However, one should note that in the context of many typical boundary conditions (e.g. periodic and Dirichlet), the absence of external force leads to global stability (see \cite{Tem84}), which makes long-time (statistically stationary in Kolmogorov sense) turbulence impossible. Therefore, external forcing (specifically its size as well as its spacial and spectral complexity) plays a particularly important role in the energy cascade formation, that is currently far from being fully understood. 

\medskip

The first rigorous results about cascade formation in the 3D Navier-Stokes case are obtained in \cite{FMRT01a, FMRT01b, FMRT01c}. These results use a traditional Fourier approach to scales and propose a formalization of the idea of infinite-time averages by introducing the concept of statistical solutions. A different approach, still connected to the idea of time averages, was employed in \cite{VF80}. Moreover, in the (more amenable to analysis) 2D case, one can also glimpse into the dependence of long-time behavior of the solution on the nature of the driving force, which, as we noted before, has a direct connection to the question of turbulence (\cite{DFJ05, DFJ07,BFJR10}). Another approach to the question of turbulence relies on the Littlewood-Paley decomposition and the use of Besov spaces (\cite{Con97,CCFS08,OR10,CS14}, see also \cite{SF75}.)

\medskip

A novel method for the study of turbulent cascades -- dubbed \emph{$(K_1,K_2)$-averaging} -- was introduced 
by the authors in \cite{DG11}. It uses a form of a "sample" average corresponding to a covering of the domain of interest by balls of size $R$, satisfying certain optimality conditions (see (\ref{K1_con}), (\ref{K2_con})). Thus the scale in our approach is the actual physical length scale of the domain. The key observation is that near-constancy of such averages across {\em all the possible coverings} at a given scale, indicates lack of significant sign-fluctuations of the averaged quantity near that particular scale. Mathematically, this approach is amenable to the use of local-in-space estimates on individual balls in the cover, which we connect to global quantities via averaging. It is flexible enough to treat specific subregions of the fluid where Fourier approach is not feasible. Moreover, it can bring spacial complexity and anisotropy into the picture: for example, to investigate the effects of vorticity coherence on the enstrophy transfer in 3D, as well as to detect
the range of scales of formation and persistence of vortex filaments
(\cite{DG13, DG12c}). As an additional benefit, this approach provides rigorous justification of validity of sampling averaging measurement techniques in turbulent fluids.

\medskip

The method was utilized in \cite{DG11} to obtain a set of sufficient conditions for energy cascades (for decaying turbulence) in the homogeneous (i.e. unforced) NSE, which are consistent with the ones obtained in \cite{FMRT01a,FMRT01b}, as well as to establish locality of such cascades. Scale-locality
is still an open question in the Fourier setting, although important progress has been made 
in \cite{CCFS08} where the authors established quasi-locality of the flux
in the Littlewood-Paley setting (see also \cite{Eyi05} for a different approach to locality). Another instance
in which $(K_1,K_2)$-averaging proved fruitful is a manifestation of dissipation anomaly, where, in the inviscid limit, the Navier-Stokes solutions displaying the ever wider inertial ranges converge to a "dissipative" (i.e. manifesting anomalous energy dissipation) solution of the Euler equations, for which the energy cascade continues \emph{ad infinitum} (\cite{DG12a, DG12b}). 

\medskip

The present paper extends the results of \cite{DG11} to the non-homogeneous, forced case. The presence of the body force presents a problem due to the lack of adequate small-scale estimates on the localized forcing term $f\cdot u$. One of the key results of this paper is that the following relation, discovered experimentally by Richardson \cite{Ric26} (see also \cite{Tay35}),  more commonly known as "Kolmogorov dissipation law", 
\begin{equation}\label{K_diss}
\varepsilon\sim\frac{U^3}{L},
\end{equation}
where $U$ is the average speed, and $L$ is the characteristic scale, plays a central role in cascade formation in physical scales settings (see Theorem 1), allowing us to bound the force-related terms and close the estimates.

\medskip

If we assume more structure, for example, space-periodicity of the solutions, we can show that crucial scaling properties must hold for a turbulent flow (see Theorem \ref{scaling} and Remark \ref{rr2}). Namely, $\varepsilon$ scales like ${\mathrm{Gr}}^{3/2}$, while the kinetic energy scales like ${\mathrm{Gr}}$, where ${\mathrm{Gr}}$ is the Grashof Number, representing non-dimensional magnitude of the force (\ref{Gr-def}), (\ref{Gr-per}). These types of scaling laws were discovered in the context of statistical solutions in \cite{DFJ09}, and the fact that {\em the same scaling} is manifested in the framework of $(K_1,K_2)$-averages points to a remarkable consistency between the two approaches. We exploit this scaling to obtain a saturation property for Cauchy-Schwarz-type inequality for averages of $f\cdot u$, which, in turn, allows us to formulate a better sufficient condition for energy cascades, as well as to estimate the width of the inertial range in terms of  Grashof number (see Theorem 6 and the remarks after thereafter) 

\medskip

We finish the paper by pointing out several links between the notions of suitable and Leray-Hopf weak solutions in order to make the connection between the approach using Fourier scales and statistical solutions (infinite-time averages), and our approach using physical scales and $(K_1,K_2)$-averages more evident.


\section{Preliminaries.}

\subsection{Suitable weak solutions.} Consider 3D incompressible NSE in a certain domain $\Omega\subseteq\mathbb{R}^3$ driven by 
an external body force $f$

\begin{equation}\label{NSE}
\begin{aligned}
\partial_t u -\nu\Delta u + (u\cdot\nabla)  u + \nabla p = f\;\\
\nabla\cdot u=0
\end{aligned}
\end{equation}

For the most part, we will consider {\em suitable weak solutions} of (\ref{NSE}), the notion first introduced by Scheffer (see \cite{Sch77,CKN82,Lem02}.) These are weak (distributional) solutions that satisfy in particular the local energy inequality:

\begin{equation}\label{ener_ineq}
\iint\left(\frac{|u|^2}{2}+p\right)\, u\cdot\nabla\phi \ge
\nu\iint |\nabla u|^2\phi
-\iint\frac{|u|^2}{2}\left(\partial_t\phi+\nu\Delta\phi\right)
-\iint f\cdot u\, \phi
\end{equation}
holds for any $C_0^{\infty}$-test function $\phi(t,x)\ge 0$. 

The suitable solutions are known to exist in many standard scenarios, including the $\Omega=\mathbb{R}^3$ (with decay at infinity), space-periodic and well as other typical boundary conditions. Notably, the problem of uniqueness and regularity are still open questions in all these settings, although via a well-know regularity result, the 1-dimensional Hausdorff measure of space-time singular points of such solution is zero (\cite{CKN82}).


\subsection{$(K_1,K_2)$-averages}\label{av-sec}

To study the behavior of physical quantities in the context of suitable solutions of (\ref{NSE}) we will employ $(K_1,K_2)$-averages defined as follows.

Let $K_1$ and $K_2$ be the fixed positive constants and let $\Omega_0$ be a subdomain of $\Omega$. We generally require $\Omega_0$ to be a bounded simply-connected domain with a piecewise-smooth boundary with the distance from $\partial\Omega$ to $\partial \Omega_0$ to be bigger then or comparable the the diameter of $\Omega_0$. For simplicity,  assume 
\begin{equation}
\Omega_0=B(0,R_0)\quad\mbox{for some}\quad R_0>0\quad\mbox{and} \quad B(0,2R_0)\subseteq\Omega.
\end{equation}
We will refer to $\Omega_0$ as {\em integral domain}, and to its radius $R_0>0$ as the {\em integral scale}.

\medskip

To localize the solutions in space and time we will employ the {\em refined cut-off functions} :
 
 Let $C_0>0$ and $\delta\in(1/2,1)$ be fixed.

For each ball $B(x_0, R)$, a refined space cut-off is a function $\psi(x)$ such that $0\le\psi\le1$,  Supp$(\psi)\subseteq B(x_0,2R)$, $\psi=1$ on $B(x_0,R)$, and
\begin{equation}\label{cut-off}
\left|\nabla\psi(x)\right|\le\frac{C_0}{R}\,\psi^{\delta}(x),\qquad 
\left|\nabla\psi(x)\right|\le\frac{C_0}{R^2}\,\psi^{2\delta-1}(x).
\end{equation}

A refined time cut-off on $[0,T]$ is a function $\eta(t)$ such that $0\le\eta\le1$,  Supp$(\eta)\subseteq [0,T]$, and
\begin{equation}\label{eta}
\left|\eta'(x)\right|\le\frac{C_0}{T}\,\eta^{\delta}(x).
\end{equation}

Now, let $0<R<R_0$. Define a {\em $(K_1,K_2)$-covering at scale $R$} any covering of $\Omega_0$ by $n$ balls of radius $R$, $\{B(x_i, R)\}_{i=\overline{1,n}}$, such that: each $x_i\in\Omega_0$ and 
\begin{align}
&\left(\frac{R_0}{R}\right)^3\le  n \le\, K_1\,\left(\frac{R_0}{R}\right)^3\label{K1_con};\\ 
\mbox{each point in}\ \Omega_0\ &\mbox{is covered by no more then}\ K_2\ \mbox{ball} \ B(x_i,2R).\label{K2_con}
\end{align}
We call $K_1$ and $K_2$ respectively the {\em global} and {\em local multiplicities} of the covering. Note that such coverings exist if $K_1$ and $K_2$ are large enough. Moreover, $K_2$ in fact determines an upper bound on $K_1$ (the lower bound is determined by the geometry of $\mathbb{R}^3$.) For example, we can choose $K_1=K_2=8$.

\medskip

Let $\psi_0$ be a fixed global cut-off, i.e  a refined cut-off associated with $\Omega_0$. 

\medskip 

Given, a $(K_1,K_2)$-covering $\{B(x_i, R)\}_{i=\overline{1,n}}$, associate each ball $B(x_i,R)$ a refined cut-offs $\psi_i$. Notice that for all $x\in\Omega_0$,
\begin{equation}\label{cut-off_bds}
\psi_0(x)\le\sum\limits_{i=1}^n \psi_i(x)\le K_2\psi_0(x)
\end{equation}
In order to ensure compatibility of boundary elements with the global cut-off 
$\psi_0$, we will change the support assumptions for the cut-offs on the boundary of $\Omega_0$ 
so that (\ref{cut-off_bds}) actually holds for all $x\in\mathbb{R}^3$, while still satisfying bounds (\ref{cut-off}).

\medskip

If $Q$ is a physical quantity with the density $q$ (or more general, if $q$ is a distribution in $\mathcal{D}'(\Omega\times[0,T])$), then define a the {\em $(K_1,K_2)$-average of $Q$ at scale $R$ associated with $\{\psi_i\}_{i=\bar{1,n}}$ } the quantity
\begin{equation}\label{ave-def}
\langle Q\rangle_R=\frac{1}{n}\frac{1}{{R}^3}\frac{1}{T}\sum\limits_{i=1}^n(q,\eta\phi_i)\;.
\end{equation}  

\medskip

We will note that if $q$ is a nonnegative distribution, then all the averages on all scales are comparable to the global (integral scale) average:
\begin{equation}\label{positivity}
\frac{1}{K_1} Q_0\le\frac{1}{n}\left(\frac{R_0}{R}\right)^3Q_0\le\langle Q\rangle_R\le K_2\frac{1}{n}\left(\frac{R_0}{R}\right)^3Q_0\le K_2Q_0\;,
\end{equation}
where the global average $Q_0$ is
\begin{equation}
Q_0=\frac{1}{{R_0}^3}\frac{1}{T}(q,\eta\phi_0)\:.
\end{equation}

In the case of a sign-changing distribution, the averages can differ widely from scale to scale and depend largely the particular choice of the $(K_1,K_2)$-covering as well as cut-offs $\psi_i$. If for a range of scales {\em all} the averages are comparable, it means that, in statistical sense there are no significant fluctuations of the sign of $q$ on those scales. However, the averages may became uncorrelated on smaller scales.

\medskip

For example, in 1D for $K_1=K_2=3$, $N\gg 1$, if $q=M(0.5+\sin(Nx))$ on an interval $[-\pi,\pi]$ (i.e. $R_0=\pi$), the global average is $Q_0=M/2$, and on scales $R\gg 1/N$ (e.g. $R>10/N$), all the averages will be comparable to $Q_0$. However, if $R<1/(4N)$, the averages depend in an essential way on the choice of the covering: if we stack the balls in the covering to emphasize negative areas of $q$, the averages will be comparable to $-Q_0\sim -M$, while if we focus on positive areas we obtain averages comparable to $Q_0\sim M$. Of course, we can also obtain anything in between. Note that the scale on which $q$ changes sign is $\sim 1/N$.

\medskip

Let $\phi=\eta\psi$, where $\psi$ is a refined cut-off corresponding to $B(x_0,R)$. In the framework of turbulence theory we distinguish the following physical quantities (all per unit of mass on $B(x_0,R)$ over time interval $[0,T]$):

\[
e_{x_0,R,T}=\frac{1}{T}\frac{1}{R^3}\iint\frac{|u|^2}{2}\phi^{2\delta-1}\quad\makebox{localized kinetic energy},
\]

\[\varepsilon_{x_0,R,T}=\epsilon_{x_0,R,T}^{\infty}+\frac{1}{T}\frac{1}{R^3}\nu\iint |\nabla u|^2\phi\quad\makebox{localized viscous + anomalous energy dissipation},\]

\[
\Phi_{x_0,R,T}=\frac{1}{T}\frac{1}{R^3}\iint\left(\frac{|u|^2}{2}+p\right) u\cdot \nabla\phi \quad\makebox{kinetic + potential energy flux},
\]

\[
 |f|^2_{x_0,R,T} =\frac{1}{n}\frac{1}{T}\frac{1}{R^3}\iint |f|^2\phi\quad\makebox{localized square-size of the force}.
\]
(the powers associated with $\phi$ are chosen for technical reasons - see the Section \ref{sec1}).

In the notations above the local energy inequality (\ref{ener_ineq}) becomes 
\begin{equation}\label{loc_en_bal}
\Phi_{x_0,R,T}=\varepsilon_{x_0,R,T}-\frac{1}{T}\frac{1}{R^3}\iint \frac{|u|^2}{2}\left(\partial_t\phi+\nu\Delta\phi\right) - \frac{1}{T}\frac{1}{R^3} \iint f\cdot u\, \phi\;.
\end{equation}
In particular, the equality in the above is achieved by adding the anomalous dissipation $\epsilon_{x_0,R,T}^{\infty}$ -- technically, a non-negative distribution $\epsilon^{\infty}$ evaluated at $\phi$. 

\medskip

Given a $(K_1,K_2)$-cover of $\Omega_0$ and a corresponding collection of refined cut-offs 
$\{\phi_i\}_{i=\bar{1,n}}$, we have $(K_1,K_2)$-averages of energy, dissipation rate, flux, and the square-length of the force:

\[
\langle e \rangle_R=\frac{1}{n} \sum\limits_{i=1}^{n}e_{x_i,R,T}\;,\quad
\langle \varepsilon \rangle_R=\frac{1}{n} \sum\limits_{i=1}^{n} \varepsilon_{x_i,R,T}\;,\quad
\langle \Phi \rangle_R=\frac{1}{n} \sum\limits_{i=1}^{n} \Phi_{x_i,R,T}\;, \quad \langle |f|^2 \rangle_R=\frac{1}{n}\sum\limits_{i=1}^n |f|^2_{x_0,R,T}\;.\]

{\rem {\rm \ \\
\begin{enumerate}
\item[(a)] Strictly speaking $\langle e \rangle_R$  is not the average of energy density $|u|^2/2$ due to the $(2\delta-1)$-power attached to $\phi_i$. Nevertheless, this average enjoys the same positivity property as the usual averages defined by (\ref{ave-def}).
\item[(b)] The other averages -- $\langle \varepsilon \rangle_R$, $\langle \Phi \rangle_R$, and $\langle |f|^2 \rangle_R$, are precisely the $(K_1,K_2)$-averages of the corresponding distributions in the sense of (\ref{ave-def}). 
Moreover, notice that the $\varepsilon$  and $|f|^2$ are non-negative distributions,
while $\Phi$ is a sign-varying distribution $(u\cdot\nabla)\,u+\nabla p$.
\end{enumerate}
}}

\medskip

We also define the global space-time averages

\begin{equation}\label{global_en_ens}\begin{aligned}
e_0=\frac{1}{T}\frac{1}{R_0^3}\iint\frac{|u|^2}{2}\phi_0^{2\delta-1}\;,\quad
\varepsilon_{0}=\epsilon_0^{\infty}+\frac{1}{T}\frac{1}{R_0^3}\nu\iint |\nabla u|^2\phi_0\;,\\
\Phi_{0}=\frac{1}{T}\frac{1}{R_0^3}\iint\left(\frac{|u|^2}{2}+p\right) u\cdot \nabla\phi_0 \;,\quad
|f|_0^2=\frac{1}{T}\frac{1}{R_0^3}\iint{|f|^2}\phi_0\;.
\end{aligned}
\end{equation}
where $\phi_0=\eta\psi_0$ is the global cut-off, and $\epsilon_0^{\infty}$ is the anomalous dissipation corresponding to $\Omega_0$.

\medskip

On one hand, consistently with the positivity property (\ref{positivity}),

\begin{equation}\label{energy}
\frac{1}{K_1} e_0\le \frac{1}{n}\left(\frac{R_0}{R}\right)^3e_0\le\langle e\rangle_R\le K_2\frac{1}{n}\left(\frac{R_0}{R}\right)^3 e_0\le K_2e_0\;.
\end{equation}
\begin{equation}\label{enstr}
\frac{1}{K_1} \varepsilon_0\le \frac{1}{n}\left(\frac{R_0}{R}\right)^3\varepsilon_0\le\langle \varepsilon\rangle_R\le K_2\frac{1}{n}\left(\frac{R_0}{R}\right)^3 \varepsilon_0\le K_2\varepsilon_0\;,
\end{equation}
end
\begin{equation}\label{Force}
\frac{1}{K_1}  |f|^2_0\le \frac{1}{n}\left(\frac{R_0}{R}\right)^3|f|^2_0\le\langle |f|^2\rangle_R\le K_2\frac{1}{n}\left(\frac{R_0}{R}\right)^3 |f|^2_0\le K_2 |f|^2_0\;.
\end{equation}

On the other hand, $\Phi$ is an a priory sign-varying quantity, and by the above discussion, we will interpret the positivity of $\langle\Phi\rangle_R$ over a range of scales  $R$ as the fact, that, in average, the energy flows towards the lower scales over that range of scales - i.e. {\em energy cascade} takes place.

In relation to turbulence, it is convenient to use adimensional parameters as an indicator of complexity of the flow. In particular, we will use {\em Grashof number}, $\mathrm{Gr}$, which characterizes the average magnitude of the driving force and the {\em Reynolds number}, $\mathrm{Re}$, which characterizes the magnitude of the average velocity:
\begin{equation}\label{Gr-def}
\mathrm{Gr}=\frac{(|f|^2_0)^{1/2}}{\nu^2R_0^{-3}}\;,\quad \mathrm{Re}=\frac{e_0^{1/2}}{\nu R_0^{-1}}\;.
\end{equation}
Note that the expressions for $\mathrm{Gr}$ and $\mathrm{Re}$ are in fact the adimensional expressions for localized (to $\Omega_0\times[0,T]$) $L^2$-norms of $f$ and $u$ respectively. For convenience, we will also denote by $F_0$ the root-mean-square average of the force:
\begin{equation}\label{F_0}
F_0=(|f|_0^2)^{1/2}=\left(\frac{1}{T}\frac{1}{R_0^3}\iint{|f|^2}\phi_0\right)^{1/2}\;.
\end{equation}


\subsection{Leray-Hopf solutions in periodic case}\label{per-sec}

While the main result of Section \ref{sec1} is valid for any suitable weak solution, the last two sections of this paper will concentrate on  space-periodic flows. In this case we have a convenient functional formulation for (\ref{NSE}). We refer to \cite{Tem84, Tem97} for more background on the settings described below.

\medskip

Let $\Omega=[0,L]^3$, $L=4R_0$, while $f\in L^{3}(\Omega)$ is an $\Omega$-periodic, divergence-free, {\em time-independent} force.

By changing coordinates, we can consider that the solutions, $u$ and the force $f$ have zero averages, i.e.
\[\int\limits_{[0,L]^3}u=\int\limits_{[0,L]^3}f=0.\]

Define

\[
H=\{u\in L^2(\omega)\ :\ u\ \mbox{is}\ \Omega-\mbox{periodic},\ \nabla\cdot u=0,\ \int_{\Omega}u=0\},
\]
and denote by $\|\cdot\|$ and $(\cdot,\cdot)$ the $L^2$-norm and $L^2$-inner product on $\Omega$.

\medskip

For notational convenience, let $A=-\Delta$ - the Stokes operator, and $B(u,v)=P_L(u\cdot\nabla)u$, where $P_L$ is the Leray projector. Note that $(B(u,v),w)=\int_{[0,L]^3}(u\cdot\nabla v)\cdot w$, while $A$ is a positive-definite, self-adjoint unbounded operator with a compact (in $L^2$) inverse, and thus the powers of $A$ are well-defined. Moreover, the norms $\|A^{\alpha/2}\cdot\|^2$ are equivalent to $H^{\alpha}$-Sobolev norms on $\Omega$. We will denote $V=D(A^{1/2})$, and $V'$- the dual of V.

\medskip

Then we can write NSE in the following functional form on $H$:
\begin{equation}\label{functional}
u_t+\nu Au + B(u,u)=f\in H\;.
\end{equation}

In this case, one can prove existence of {\em Leray-Hopf weak solutions}, i.e, $u\in L^{\infty}([0,\infty),H)\cap L^2_{loc}((0,\infty),V)$, such that (\ref{functional}) is satisfied in $V'$:
\begin{equation}\label{weak_form}
(u,v)_t+(A^{1/2}u,A^{1/2}v)-(B(u,v),u)=(f,v)\quad\mbox{in distribution sense for all}\ v\in V\;,
\end{equation}
and for which the {\em Leray-Hopf energy inequality} holds:

\begin{equation}\label{leray_en_ineq}
\nu\int\limits_{t_0}^{t_1} \|\nabla u\|^2\le \frac{\|u(t_0)\|^2}{2}-\frac{\|u({t_1})\|^2}{2} + \int\limits_{t_0}^{t_1}  (f,u)\qquad \mbox{a.e. in}\ t_0\ \mbox{and for all}\ {t_1}\ge t_0 \;.
\end{equation}

\medskip

In particular, the Leray-Hopf solutions are regular (i.e. in $D(A)$) for all times, except possibly on a closed set of 1-dimensional Hausdorff measure zero. Therefore, (\ref{leray_en_ineq}) becomes equality on a dense set of disjoint open time-intervals in $[0,\infty)$.

\medskip

Another peculiar feature of Leray-Hopf solutions is the existence of the {\em weak global attractor}, $\mathcal{A}_w$ (cf. \cite{FT87,Che09,FRT10}). The set $\mathcal{A}_w\subset H$ can be defined as the set of all global (forward and backwards) in time Leray-Hopf solutions. one can show that $\mathcal{A}_w$ is bounded in $H$ and weakly attracts bounded sets as $t\to\infty$.

\medskip
It is known (see e.g. \cite{Lem02}) that the solutions originally constructed by Leray (see \cite{Ler34}), will satisfy both local energy inequality (\ref{ener_ineq}) (i.e. they are suitable solutions in \cite{Sch77} and \cite{CKN82} sense) as well as the global Leray energy inequality (\ref{leray_en_ineq}).

We note that the argument given in \cite{Lem02} is in whole $\mathbb{R}^3$, however, similar result holds in the periodic case, in fact  one can prove that Leray's method provides solutions that satisfy the following version of  the Leray-Hopf energy inequality:

\begin{equation}\label{L-H_ener_ineq}
\nu\int\limits_0^T\eta\, \|\nabla u\|^2\le \int\limits_0^T \eta_t\,\frac{\|u\|^2}{2} + \int\limits_0^T \eta\, (f,u)\;,
\end{equation}
where $\eta$ is a refined cut-off on $[0,T]$ provided by (\ref{eta}). 

\medskip

In this settings, we will find more convenient to use a slightly modified -- non-localized -- version of Grashof number:
\begin{equation}\label{Gr-per}
\mathrm{Gr}=\frac{\|f\|}{\nu^2 R_0^{-3/2}}\;.
\end{equation}

As we show in Appendix (section \ref{appx}), {\em any} Leray-Hopf solution will actually satisfy  (\ref{L-H_ener_ineq}).
Moreover, {\em any suitable weak solution} in the periodic case in addition to the local energy inequality (\ref{ener_ineq}) will necessarily satisfy (\ref{L-H_ener_ineq}) as well.

\medskip

Thus, in Section \ref{periodic_casc}, when talking talking about periodic solutions, we will assume that we are dealing with solutions that satisfy both local and global energy inequalities (\ref{ener_ineq}) and (\ref{L-H_ener_ineq}). We note that the results of Section \ref{en-enst-sec} do not require that solutions are suitable in Scheffer (\cite{Sch77}) sense, i.e. 
(\ref{ener_ineq}) is not assumed.


\section{Energy Cascade Theorem}\label{sec1}

Assume $f\in L^2(\Omega\times[0,T])$ is a driving force in the context of suitable solutions to NSE (see (\ref{ener_ineq})).
Also, let $\{\phi_i\}_{i=\overline{1,n}}$, $\phi_i=\eta\psi_i$, be a family of refined cut-offs corresponding to a $(K_1,K_2)$-average (see Section \ref{av-sec}). 

\medskip

For convenience, assume 

\begin{equation}\label{T_assum}
T\ge\frac{R_0^2}{\nu}.
\end{equation}

Also denote 
\[e_i=e_{x_i,R,T},\quad \varepsilon_i=\varepsilon_{x_i,R,T},\quad \mbox{and}\quad\Phi_i=\Phi_{x_i,R,T}\;.\]
Then the local energy balance (see (\ref{loc_en_bal})) is

\begin{equation}\label{loc_en_bal0}
\Phi_i=\varepsilon_i-\frac{1}{T}\frac{1}{R^3}\iint \frac{|u|^2}{2}\left(\partial_t\phi_i+\nu\Delta\phi_i\right) - \frac{1}{T}\frac{1}{R^3} \iint f\cdot u\, \phi_i\;.
\end{equation}

Recall, (\ref{enstr}) states:
\begin{equation}\label{enstr1}
 \frac{1}{n}\left(\frac{R_0}{R}\right)^3\varepsilon_0\le\langle \varepsilon\rangle_R\le K_2\frac{1}{n}\left(\frac{R_0}{R}\right)^3 \varepsilon_0\;.
\end{equation}

Use (\ref{cut-off}) and (\ref{eta}) to estimate:

\begin{equation}\label{en1}
\left|
\frac{1}{T}\frac{1}{R^3}\iint \frac{|u|^2}{2}\left(\partial_t\phi_i+\nu\Delta\phi_i\right) 
\right|\le C_0\left(\frac{1}{T}+\frac{\nu}{R^2}\right) \frac{1}{T}\frac{1}{R^3}\iint \frac{|u|^2}{2}\eta^{\delta}\psi_i^{2\delta-1}\le 2\,C_0 \frac{\nu}{R^2}\, e_i\, ,
\end{equation}
and thus, by (\ref{energy}),
\begin{equation}\label{en2}
\left| \langle \iint \frac{|u|^2}{2}\left(\partial_t\phi_i+\nu\Delta\phi_i\right) \:\rangle_R\right|\le 2\,C_0K_2 \frac{\nu}{R^2}\, \frac{1}{n}\left(\frac{R_0}{R}\right)^3e_0\, .
\end{equation}

We bound the force term by:

\[
\left|
\frac{1}{T}\frac{1}{R^3} \iint f\cdot u\, \phi_i
\right| 
\le
\sqrt{2} \frac{1}{T}\frac{1}{R^3}\left(\iint |f|^2\phi_i^{3-2\delta}\right)^{1/2}\left(\iint\frac{|u|^2}{2}\phi_i^{2\delta-1}\right)^{1/2}=\sqrt{2} \left(\frac{1}{T}\frac{1}{R^3}\iint |f|^2\phi_i^{3-2\delta}\right)^{1/2}\, e_i^{1/2}\;,
\]
and consequently
\[
\begin{aligned}
\left| \langle\: (f, u) \:\rangle_R\right|  &\le 
\sqrt{2}\, \frac{1}{n}\sum\limits_{i=1}^n \left(\frac{1}{T}\frac{1}{R^3}\iint |f|^2\phi_i^{3-2\delta}\right)^{1/2}\, e_i^{1/2} \\
&\le
\sqrt{2} \frac{1}{n} \left(\frac{1}{T}\frac{1}{R^3}\sum\limits_{i=1}^n\iint |f|^2\phi_i^{3-2\delta}\right)^{1/2} \left(\sum\limits_{i=1}^n e_i\right)^{1/2}\\
&\le  
\sqrt{2}\,
 \left(\frac{1}{n}\frac{1}{T}\frac{1}{R^3}\sum\limits_{i=1}^n\iint |f|^2\phi_i^{3-2\delta}\right)^{1/2} \left(\frac{1}{n}\sum\limits_{i=1}^n e_i\right)^{1/2}\\
 &\le\sqrt{2}\,\langle |f|^2\rangle_R^{1/2}\langle e \rangle_R^{1/2}\;,\\
\end{aligned}
\]
where in the last inequality we use that $\phi_i^{3-2\delta}\le\phi_i$ (since $3-2\delta>1$).
Thus, by (\ref{Force}), we have 
\[
\langle |f|^2 \rangle_R\le K_2\,\frac{1}{n}\left(\frac{R_0}{R}\right)^3{F_0^2}\;,
\]
and consequently,
\begin{equation}\label{force}
\left| \langle\: (f,u) \:\rangle_R\right|  \le \sqrt{2}\,K_2\, \frac{1}{n}\left(\frac{R_0}{R}\right)^3 F_0\, e_0^{1/2}\;.
\end{equation}

Next, taking the average in (\ref{loc_en_bal0}), using the bounds (\ref{enstr1}), (\ref{en1}), and (\ref{force}), we obtain:

\begin{equation}\label{ava_bounds}
\frac{1}{n}\left(\frac{R_0}{R}\right)^3 \left[\varepsilon_0 - 2C_0K_2\frac{\nu}{R^2}\, e_0 - \sqrt{2}K_2 {F_0}\, e_0^{1/2} \right]
\le \langle \Phi\, \rangle_R \le
 K_2\, \frac{1}{n}\left(\frac{R_0}{R}\right)^3\left[
\varepsilon_0 +2C_0\frac{\nu}{R^2}\, e_0 +  \sqrt{2}F_0\, e_0^{1/2}\right]\:. 
\end{equation}

Notice that (\ref{K1_con}) implies
\begin{equation}\label{R-n}
\frac{1}{K_1}\le\frac{1}{n}\left(\frac{R_0}{R}\right)^3\le 1\;,
\end{equation}
and so we obtain
\[
\frac{1}{K_1}\ \left[\varepsilon_0 - 2C_0K_2\frac{\nu}{R^2}\, e_0 -\sqrt{2} K_2 F_0\, e_0^{1/2} \right]
\le \langle \Phi\, \rangle_R \le
 K_2\, \left[
\varepsilon_0 + 2C_0\frac{\nu}{R^2}\, e_0 +  \sqrt{2}F_0\, e_0^{1/2}\right]
\]

Assume:
\[ \sqrt{2}\,K_2\, F_0\, e_0^{1/2}\le \frac{1}{2} \varepsilon_0\:,\]

Then

\[
\frac{1}{K_1}\left[\frac{1}{2}\varepsilon_0-2C_0K_2\frac{\nu}{R^2}e_0\right]\le \langle \Phi\, \rangle_R \le
K_2\left[ \frac{3}{2}\varepsilon_0+2C_0\frac{\nu}{R^2}e_0\right]\;,
\]
which implies
\[
\frac{1}{2K_1}\varepsilon_0\left(1-{4C_0K_2}\frac{\nu e_0/\varepsilon_0}{R^2}\right)
\le \langle \Phi\, \rangle_R \le
\frac{3K_2}{2}\varepsilon_0\left(1+\frac{4C_0}{3}\frac{\nu e_0/\varepsilon_0}{R^2}\right)\;,
\]
and consequently,
\begin{equation}
\frac{1}{2K_1}\varepsilon_0\left(1-\alpha\frac{\tau_0^2}{R^2}\right)
\le \langle \Phi\, \rangle_R \le
\frac{3K_2}{2}\varepsilon_0\left(1+\alpha\frac{\tau_0^2}{R^2}\right)\:,
\end{equation}
where
\begin{equation}
\alpha=4C_0K_2
\end{equation}
and
\begin{equation}\label{tau_0}
\tau_0=\left(\frac{\nu e_0}{\varepsilon_0}\right)^{1/2}\quad \mbox{is the Taylor scale.}
\end{equation}

If $\alpha\tau_0^2/R_0^2<1/2$ (or, equivalently $8C_0K_0\, \nu e_0/R_0^2\le\varepsilon_0$), the inequality above implies the following theorem:

\begin{thm}\label{casc-thm}
Let (\ref{T_assum}) be satisfied and denote
 \begin{equation}\label{alpha_0}
\alpha_0=\sqrt{8}K_2\quad\mbox{and}\quad \beta_0=8C_0K_2\;.
\end{equation}
Assume
\begin{equation}\label{ass1}
\alpha_0 F_0\,e_0^{1/2}\le\varepsilon_0,
\end{equation}
and
\begin{equation}\label{ass2} 
\beta_0\nu \frac{e_0}{R_0^2}\le\varepsilon_0
\end{equation}
hold. Then we have energy cascade
\begin{equation}
\frac{1}{4K_1}\, \varepsilon_0\le \langle \Phi\, \rangle_R \le \frac{9K_2}{4}\, \varepsilon_0
\end{equation}
for all $R$ inside the inertial range
\begin{equation}\label{inert-range}
\left[
\sqrt{\beta_0}\, \tau_0,\; R_0
\right]\;.
\end{equation}

\end{thm}

Notice that  the condition (\ref{ass2}) is essentially the same as in the cascade theorem for the unforced case, see \cite{DG11}. It requires that the Taylor length scale $\tau_0$ is smaller then the integral scale $R_0$. The other condition, (\ref{ass1}), contains dependence on the size of the force (automatically satisfied if $f=0$), which requires the energy dissipation rate to be big enough compared to both the magnitude of $f$ and $e_0$. As we will see in the periodic case, this condition amounts to saturation of Kolmogorov dissipation law. Note that the width of the inertial range (\ref{inert-range}) is not directly dependent on (\ref{ass1}).


\section{Bounds on average energy and enstrophy}\label{en-enst-sec}

Consider the 3D incompressible NSE with periodic boundary conditions inside $\Omega=[0,L]^3$,  driven by an $L^2$, divergence-free, time-independent force $f$ - see Section \ref{per-sec}. Set the integral scale $R_0=L/4$.

Our goal is to study the kinetic energy and its dissipation rate on the {\em whole} periodic box $\Omega$. As we noted before,  in this section we will only require that $u$ is a Leray-Hopf solution, and as consequence (see Theorem \ref{str_en_ineq} and the Remark that follows), satisfy the energy inequality (\ref{L-H_ener_ineq}). Thus, the results would hold even for (hypothetical) Leray-Hopf solutions that are not suitable in Scheffer  sense (\cite{Sch77}).

\subsection{A consequence of the energy inequality}\label{sec2.1}

We can take advantage of periodicity to simplify definitions for global energy and enstrophy used in sections \ref{av-sec} and \ref{sec1}.

\medskip

Namely, define the (modified)  global energy
\begin{equation}\label{periodic_global-en}
e_0=\frac{1}{T}\frac{1}{R_0^3}\int\limits_0^T\eta^{\delta}\,\frac{\|u\|^2}{2}\;,
\end{equation}
and the total (viscous+anomalous) dissipation rate:
\begin{equation}\label{periodic_global-enst}
\varepsilon_0=\frac{1}{T}\frac{1}{R_0^3}\nu\int\limits_0^T\eta\, \|\nabla u\|^2\, + \varepsilon_{\infty},
\end{equation}
where the anomalous dissipation $\epsilon_{\infty}\ge 0$ is such that
\begin{equation}\label{en-balance}
\varepsilon_0=\frac{1}{T}\frac{1}{R_0^3}\int\limits_0^T \eta_t\,\frac{\|u\|^2}{2} +\frac{1}{T}\frac{1}{R_0^3} \int\limits_0^T \eta\, (f,u)\;.
\end{equation}
 
We bound
\[
\left|\frac{1}{T}\frac{1}{R_0^3}\int_0^T\eta_t\, \frac{\|u\|^2}{2}\right| \le \frac{C_0}{T}e_0,
\]
\[
\left|\frac{1}{T}\frac{1}{R_0^3}\int_0^T \eta\, (f,u) \right|\le \frac{1}{T}\frac{1}{R_0^3} \left(\int_0^T \eta^{2-\delta}\|f\|^2\right)^{1/2}\left(\int\limits_0^T \eta^{\delta}\|u\|^2\right)^{1/2}\le \sqrt{2}\frac{\|f\|}{R_0^{3/2}}\,e_0^{1/2}.
\]

Use these in (\ref{en-balance}) to obtain

\begin{equation}\label{Eef1}
\varepsilon_0\le C_0\frac{e_0}{T}+\sqrt{2}\frac{\|f\|}{R_0^{3/2}} e_0^{1/2}
\end{equation}

If $u\in\mathcal{A}_w$, i.e. $u$ is on the weak attractor (cf. \cite{FT87}), then
\[\|u\|^2\le \left(\frac{2}{\pi}\right)^4\frac{R_0^4}{\nu^2}\,\|f\|^2,\]
and consequently,
\[e_0\le \frac{8}{\pi^4}\frac{R_0}{\nu^2}\|f\|^2\;.\]
Therefore on the weak attractor
\begin{equation}\label{e-bound}
C_0\frac{e_0}{T}=\frac{C_0}{T} e_0^{1/2}e_0^{1/2}\le\frac{\sqrt{8}C_0}{\pi^2}\frac{1}{T}\frac{R_0^2}{\nu}\frac{\|f\|}{R_0^{3/2}} e_0^{1/2},
\end{equation}
and so (\ref{Eef1}) yields the allowing.

\begin{thm}\label{Eef-thm}
Assume that $u(t)\in\mathcal{A}_w$ and 
\begin{equation}\label{T1}
T\ge\frac{2C_0}{\pi^2}\frac{R_0^2}{\nu}\;.
\end{equation}
Then
\begin{equation}\label{Eef0}
\varepsilon_0\le 2\sqrt{2}\frac{\|f\|}{R_0^{3/2}} e_0^{1/2}\;.
\end{equation}
\end{thm}

{\rem\label{rem1}{\rm 
Notice that since $K_2>1$, the bound (\ref{Eef0}) is incompatible with the assumption (\ref{ass1}), which means that Theorem \ref{casc-thm} cannot be applied in this particular case if the integral domain is the whole periodic box $\Omega$. Of course, the assumptions of the cascade theorem might still hold on a small enough subregion of $\Omega$. In section \ref{periodic_casc} we will modify conditions (\ref{ass1}) and (\ref{ass2}) to ensure a cascade on the global domain.
}}

\subsection{A lower bound on energy.}

In fluid dynamics, it is natural to assume that the bigger is the driving force, the bigger is the velocity (and the kinetic energy), and hence the Reynolds number of the flow. In this section our goal is to obtain the lower bound on the kinetic energy $e_0$ in periodic setting described above. The idea behind the bounds presented here is similar to \cite{DFJ09}, where the lower bound on energy was obtained in a different context.

\smallskip

In what follows, we assume $f\in H^1$ (or $f\in D(A^{1/2})$). 

\smallskip

Then, by duality, multiplying NSE by $\eta A^{-1}f$, and integrating by parts, keeping in mind the divergence-free condition, we obtain:
\begin{equation}\label{eq1}
-\int\limits_0^{T}\eta_t (u,A^{-1}f) + \nu \int\limits_0^T \eta\, (u,f) - \int\limits_0^T \eta\, (B(u,A^{-1}f),u) =
\int\limits_0^T \eta\, \|A^{-1/2}f\|^2\;.
\end{equation}
As before, we bound
\[
\left|\frac{1}{T}\frac{1}{R_0^3}\int_0^T \eta_t (u,A^{-1}f)\,\right|\le \frac{C_0}{T}\frac{1}{T}\frac{1}{R_0^3}\int_0^T \eta^{\delta}\|u\|\,\|A^{-1}f\|\le \frac{\sqrt{2}C_0}{T}\frac{\|A^{-1}f\|}{R_0^{3/2}} e_0^{1/2}\;,
\] 
\[
\left|\frac{1}{T}\frac{1}{R_0^3}\nu \int_0^T \eta\, (f,u)\,\right| \le \sqrt{2}\nu \frac{\|f\|}{R_0^{3/2}} e_0^{1/2}\:
\]
and use the Agmon inequality $\|w\|_{\infty}\le C_A \|A^{1/2}w\|^{1/2}\|Aw\|^{1/2}$ to bound the trilinear term:
\[
\left|\frac{1}{T}\frac{1}{R_0^3}\int_0^T \eta\, (B(u,A^{-1}f),u)\,\right|\le
\frac{1}{T}\frac{1}{R_0^3}\int_0^T \eta\, C_A  \|f\|^{1/2}\|A^{1/2}f\|^{1/2}\|u\|^2=
2C_A  \|f\|^{1/2}\|A^{1/2}f\|^{1/2}\, e_0\;.
\]
Finally,
\[
\frac{1}{T}\frac{1}{R_0^3}\int\limits_0^T \eta\, \|A^{-1/2}f\|^2=
c_{\eta}  \frac{\|A^{-1/2}f\|^2}{R_0^3}\;,
\]
where $c_\eta\in(0,1)$ is defined by
\[c_\eta=\frac{1}{T}\int\limits_0^T\eta\]

Then, (\ref{eq1}) implies:
\begin{equation}\label{eq2}
c_{\eta} \frac{\|A^{-1/2}f\|^2}{R_0^3}\le \sqrt{2}\left( \frac{C_0}{T}\frac{\|A^{-1}f\|}{\|f\|} + \nu \right)\, \frac{\|f\|}{R_0^{3/2}} e_0^{1/2} + 2C_A  \|f\|^{1/2}\|A^{1/2}f\|^{1/2}\, e_0\;.
\end{equation}

Note that if (\ref{T1}) holds, i.e,
\[T\ge \frac{2C_0}{\pi^2}\frac{R_0^2}{\nu}\;,\]
then, using the Poincar\'e inequality $\|A^{-1}f\|\le \left({2/\pi}\right)^2 {R_0^2}\,\|f\|$,
\[
\frac{C_0}{T}\frac{\|A^{-1}f\|}{\|f\|} + \nu\le \frac{\pi^2}{2} \frac{\nu}{R_0^2} \left(\frac{2}{\pi}\right)^2 {R_0^2} + \nu =3\nu.
\]
For now, we will denote 
\[\gamma_f= \frac{C_0}{\nu T}\frac{\|A^{-1}f\|}{\|f\|} + 1,\]
remembering that when (\ref{T1}) holds, then $\gamma_f\le 3$. Note however, that for any $T$, by choosing rough enough $f$, we can make $\gamma_f$ as close to 1 as needed.

We can view (\ref{eq2}) as a quadratic inequality in $e_0^{1/2}$. Thus, we solve
\begin{equation}\label{eq3}
e_0^{1/2}\ge\frac{-\sqrt{2}\gamma_f \nu\|f\|/R_0^{3/2}+\sqrt{2\gamma_f^2 \nu^2\|f\|^2/R_0^3+8C_A\, \|f\|^{1/2}\|A^{1/2}f\|^{1/2}c_{\eta}\|A^{-1/2}f\|^2/R_0^3     }     }{4C_A\, \|f\|^{1/2}\|A^{1/2}f\|^{1/2}}
\end{equation}

Note that for $g,h>0$ and $\alpha\in(0,1)$,
\[-h+\sqrt{h^2+g^2}\ge\,\alpha\, g\quad\Leftrightarrow\quad g\ge\frac{2\alpha}{1-\alpha^2} h\;.\]
Apply the equivalence above to (\ref{eq3}) with $\alpha=1/\sqrt{2}$, $g^2=8c_{\eta}C_A\, \|f\|^{1/2}\|A^{1/2}f\|^{1/2}\|A^{-1/2}f\|^2$ and $h^2=2\gamma_f^2 \nu^2\|f\|^2$ to obtain

\begin{equation}
e_0^{1/2}\ge \frac{\frac{1}{\sqrt{2}}\,g}{4R_0^{3/2}C_A\, \|f\|^{1/2}\|A^{1/2}f\|^{1/2}}=
\frac{1}{{2}} \frac{c_{\eta}^{1/2}}{C_A^{1/2}}\frac{1}{R_0^{3/2}}\frac{\|A^{-1/2}f\|}{\|f\|^{1/4}\|A^{1/4}f\|^{1/4}}\:,
\end{equation}
provided
\begin{equation}
8c_{\eta}C_A\,\|f\|^{1/2}\|A^{1/2}f\|^{1/2}\|A^{-1/2}f\|^2\ge (2\sqrt{2})^2 2\gamma_f^2\nu^2\|f\|^2
\end{equation}

We collect the result above in the following theorem.

\begin{thm}\label{lb-e-thm}
Assume 
\begin{equation}\label{F-ass}
\|f\|\ge \frac{2\gamma_f^2\nu^2}{c_{\eta}C_A}\frac{\|f\|^{1/2}}{\|A^{1/2}f\|^{1/2}}\frac{\|f\|^2}{\|A^{-1/2}f\|^2}\;,
\end{equation}
where 
\begin{equation}
\gamma_f= \frac{C_0}{\nu T}\frac{\|A^{-1}f\|}{\|f\|} + 1\quad\mbox{and}\quad c_\eta=\frac{1}{T}\int\limits_0^T\eta
\end{equation}
Then
\begin{equation}\label{lb-e}
e_0\ge \frac{c_{\eta}}{4C_A}\frac{1}{R_0^{5/2}} \frac{\|A^{-1/2}f\|^2}{\|f\|^{3/2}\|A^{1/2}f\|^{1/2}}\, \frac{\|f\|}{R_0^{1/2}}\;.
\end{equation}

\end{thm}

\subsection{One-sided Kolmogorov's dissipation law.}

Suppose now that the assumptions of both Theorem \ref{Eef-thm} and Theorem \ref{lb-e-thm} hold, i.e.
$f\in D(A^{1/2})$, $u\in\mathcal{A}_w$, and (\ref{T1}) and (\ref{F-ass}) hold.

Then,

\[
\varepsilon_0\le \frac{2\sqrt{2}}{R_0^{3/2}}e_0^{1/2}\, \|f\|\le 
\frac{2\sqrt{2}}{R_0^{3/2}}e_0^{1/2}R_0^{1/2}
\frac{4C_A}{c_{\eta}}{R_0^{5/2}} \frac{\|f\|^{3/2}\|A^{1/2}f\|^{1/2}}{\|A^{-1/2}f\|^2} e_0=
\frac{8\sqrt{2}C_A}{c_{\eta}} R_0^{5/2} \frac{\|f\|^{3/2}\|A^{1/2}f\|^{1/2}}{\|A^{-1/2}f\|^2}\, \frac{e_0^{3/2}}{R_0}.
\]

Thus we have the following result.

\begin{thm}\label{apriori_bds}
Assume $f\in D(A^{1/2})$, $u\in \mathcal{A}_w$, 
\begin{equation}\label{ass3}
\|f\|\ge \sigma_f\quad\mbox{and}\quad T\ge c_1\frac{R_0^2}{\nu}\;.
\end{equation}
Then:
\begin{equation}\label{lkd1}
\varepsilon_0\le 2\sqrt{2}\frac{\|f\|}{R_0^{3/2}}\, e_0^{1/2}\;,
\end{equation}
\begin{equation}\label{lkd2}
e_0\ge \theta_f \frac{\|f\|}{R_0^{1/2}}\;,
\end{equation}
and
\begin{equation}\label{lkd3}
\varepsilon_0\le \frac{2\sqrt{2}}{\theta_f} \frac{e_0^{3/2}}{R_0},
\end{equation}
where
\[c_1=\frac{2C_0}{\pi^2}\;,\quad \sigma_f=\frac{2\gamma_f^2\nu^2}{c_{\eta}C_A}\frac{\|f\|^{1/2}}{\|A^{1/2}f\|^{1/2}}\frac{\|f\|^2}{\|A^{-1/2}f\|^2}\;,\quad \theta_f=\frac{c_{\eta}}{4C_A}\frac{1}{R_0^{5/2}} \frac{\|A^{-1/2}f\|^2}{\|f\|^{3/2}\|A^{1/2}f\|^{1/2}}\;.\]
\end{thm}

As can be seen from Theorem \ref{apriori_bds}, the classical Kolmogorov dissipation law (\ref{K_diss}), an important part of turbulent phenomenology, can be viewed as saturation of the a priori bound (\ref{lkd3}). The non-dimensional shape factors of the force $\sigma_f$ and $\theta_f$ while invariant under scaling, can be small for rough enough $f\in D(A^{1/2})$.

 We note that the rigorous on-sided bounds in Kolmogorov's dissipation law were obtained previously in various settings \cite{CD94,Foi97,DF02,DFJ09}. The bound (\ref{lkd3}) is specific to our, time-localized case.

In the next section we will investigate several consequences of saturation of this bound. 

\subsection{Turbulent scaling}

Assume that the solution $u$, in addition to the conditions of Theorem \ref{apriori_bds},  saturates  (\ref{lkd3}), i.e. Kolmogorov's dissipation law holds:
\begin{equation}\label{satur1}
\mbox{There exists}\ K>0\ \mbox{such that}\quad  K\,\frac{e_0^{3/2}}{R_0}\le\varepsilon_0\le \frac{2\sqrt{2}}{\theta_f} \frac{e_0^{3/2}}{R_0}.
\end{equation}

Then we can prove the following theorem.

\begin{thm}\label{scaling}
Assume (\ref{ass3}) and (\ref{satur1}) hold. Then
\begin{equation}\label{scal1}
\frac{K^{3/2}\theta_f^{3/2}}{8^{1/4}}
\frac{\|f\|}{R_0^{3/2}}\, e_0^{1/2}\le\varepsilon_0\le 2\sqrt{2}\frac{\|f\|}{R_0^{3/2}}\, e_0^{1/2}\;,
\end{equation}
\begin{equation}\label{scal2}
 \theta_f \frac{\|f\|}{R_0^{1/2}}\le e_0\le \frac{2\sqrt{2}}{K}\frac{\|f\|}{R_0^{1/2}}\;,
\end{equation}
and
\begin{equation}\label{scal3}
 K \theta_f^{3/2}\frac{\|f\|^{3/2}}{R_0^{7/4}}\le\varepsilon_0\le \frac{8^{5/4}}{K^{3/2}\theta_f} \frac{\|f\|^{3/2}}{R_0^{7/4}}\;.
\end{equation}
\end{thm}

\begin{proof}
Using (\ref{lkd1}) together with (\ref{satur1}) we obtain:
\[
K  \frac{e_0^{3/2}}{R_0}\le 2\sqrt{2}\frac{\|f\|}{R_0^{3/2}}\, e_0^{1/2},
\]
and thus
\[
e_0\le\frac{2\sqrt{2}}{K}\,\frac{\|f\|}{R_0^{1/2}},
\]
so (\ref{scal2}) holds.

Using (\ref{scal2}) together with (\ref{satur1}), we obtain
\[
K\frac{\theta_f^{3/2}\frac{\|f\|^{3/2}}{R^{3/4}} }{R_0}\le \varepsilon_0 \le \frac{2\sqrt{2}}{\theta_f}\frac{  \left(\frac{2\sqrt{2}}{K}\right)^{3/2} \frac{\|f\|^{3/2}}{R^{3/4}} }{R_0},
\]
and (\ref{scal3}) follows.

Finally, using (\ref{scal3}) together with (\ref{scal2}) we obtain:
\[
\frac{K^{3/2}\theta_f^{3/2}}{8^{1/4}}
\frac{\|f\|}{R_0^{3/2}}\, e_0^{1/2}
\le\frac{K^{3/2}\theta_f^{3/2}}{8^{1/4}}
\frac{\|f\|}{R_0^{3/2}}
 \left(\frac{2\sqrt{2}}{K}\frac{\|f\|}{R_0^{1/2}}\right)^{1/2}
\le K \theta_f^{3/2}\frac{\|f\|^{3/2}}{R_0^{7/4}}\le\varepsilon_0\;,
\]
and thus (\ref{scal1}) holds.

\end{proof}

{\rem\label{rr2}{\rm 
If  $K=c/\theta_f$, with $0<c<2\sqrt{2}$, i.e. inequality (\ref{satur1}) is "fully" saturated, then (\ref{scal1})--(\ref{scal3}) become:
\[
\frac{c^{3/2}}{8^{1/4}}\frac{\|f\|}{R_0^{3/2}}\, e_0^{1/2}\le\varepsilon_0\le 2\sqrt{2}\frac{\|f\|}{R_0^{3/2}}\, e_0^{1/2}\;,
\]
\[
 \theta_f \frac{\|f\|}{R_0^{1/2}}\le e_0\le \frac{2\sqrt{2}}{c}\theta_f\,\frac{\|f\|}{R_0^{1/2}}\;,
\]
and
\[
c \theta_f^{1/2}\frac{\|f\|^{3/2}}{R_0^{7/4}}\le\varepsilon_0\le \frac{8^{5/4}}{c^{3/2}}\theta_f^{1/2} \frac{\|f\|^{3/2}}{R_0^{7/4}}\;.
\]
}}

{\rem\label{rem4}{\rm
In fact, proceeding similarly to the proof above, one can show that saturation of any one of the inequalities (\ref{lkd1})--(\ref{lkd3}), brings about the same conclusions as Theorem  \ref{scaling}, with slightly different absolute constants (the parts of the constants depending on the shape of force will be the same.)
}}

{\rem\label{rem04}{\rm
Theorem {\ref{scaling}} is reminiscent of the main results of \cite{DFJ09}. There the framework of Fourier scales and stationary statistical solutions was used to show that the saturation of Kolmogorov's dissipation law induces much similar scalings on energy and its dissipation rate, namely $\sim {\rm Gr}$ and $\sim {\rm Gr}^{3/2}$ respectively, where ${\rm Gr}=\|f\|/\nu^2R_0^{-3/2}$ is the Grashof number (\ref{Gr-per}). In that particular framework, ${\rm Gr}\gg 1$ would {\em automatically} imply energy cascade in Fourier scales on the inertial range $[\tau_0,r_f]$, where $r_f$ is the scale of the force. The reason for this is that the aforementioned ${\rm Gr}$-scaling implies $\tau_0(=\kappa_{\tau}^{-1})\sim{\rm Gr}^{-1/4}\ll1$ (see \cite{DFJ09}), and thus the sufficient condition for cascades given in \cite{FMRT01a,FMRT01b} holds.
}}


\section{Global cascades in periodic setting.}\label{periodic_casc}

In this section  assume that $u$ is a suitable solution of the space periodic NSE in \cite{Sch77} and \cite{CKN82} sense and as before, denote $\Omega=[0,L]^3$ - the basic periodic box. Recall (see Appendix \ref{appx}), that $u$ must also be a weak solution satisfying the (localized in time) global energy inequality (\ref{L-H_ener_ineq}), and therefore the results of Section \ref{en-enst-sec} apply.

\medskip

Our goal here is to develop a sufficient condition for energy cascades, similar to Theorem \ref{casc-thm}, applicable globally on an integral domain consisting of the entire periodic box $\Omega$. 

\medskip

First, note that in this case we can modify our averages to reflect periodicity: 

\begin{enumerate}
\item The cut-off functions will be considered $L$-periodic: we can extend each $\psi_i$ periodically, provided Supp$(\psi_i)$ is inside a cubic box of size $L$ (i.e. the radius $R$ of the covering balls should be less then $R_0:=L/4$.)
\item Otherwise, on each box of size $L$ we will still require a $(K_1,K_2)$-covering. 
\item The global (space) cut-off therefore can be taken $\psi_0=1$ with no special condition on "boundary" elements of the covering.
\item To accommodate time cut-offs, we will again use the modified energy inequality (\ref{L-H_ener_ineq}).
\end{enumerate}

Thus, the energy and energy dissipation rate -- $e_0$ and $\epsilon_0$ defined in Section \ref{sec2.1} are the global energy and its dissipation rate 
defined for the integral domain $\Omega_0=\Omega$.

\medskip

Recall that in this settings equation (\ref{en-balance}) implies that Theorem \ref{casc-thm} is impossible on $\Omega$ (see Remark \ref{rem1}.)

\subsection{A better treatment of $\langle (f,u)\rangle_R$ term.}

Recall that the 1st condition of the cascade Theorem \ref{casc-thm} was  consequence of estimating 
\[
\langle (f,u)\rangle_R=\frac{1}{R^3}\frac{1}{T}\frac{1}{n}\sum\limits_{i=1}^{n}\iint f\cdot u\, \phi_i\, dxdt,
\]
where $\phi_i=\eta\psi_i$ are the space-time cut-offs.
It turns out, in the periodic case we can obtain a better bound for the averages on the global domain $\Omega$.

\begin{lem}\label{ll1}
Assume $u\in\mathcal{A}_w$ and the following conditions are satisfied:
\begin{equation}\label{ll1-a1}
\frac{2\sqrt{2}}{K} \frac{\|f\|}{R_0^{3/2}}e_0^{1/2}\le \varepsilon_0.
\end{equation}
and
\begin{equation}\label{ll1-a2}
T\ge \frac{4K C_0}{\pi^2}\frac{R_0^{2}}{\nu},
\end{equation}
where $K>1$ is a certain constant, that can depend on shape of $f$.

Then for any collection of the cut-offs $\{\psi_i\}$ associated with a $(K_1,K_2)$-average, we have

\begin{equation}\label{ll1-concl}
\left| \iint \, (\eta\sum\limits_{i=1}^{n}\psi_i)\, f\cdot u \: dxdt\right| \le  KK_2\int\limits_0^T\eta\,(f,u)\, dt.
\end{equation}

(In particular, the integral in the right-hand side must be positive, which is indeed the case - see (\ref{ll1-p1}) below.)

\end{lem}

\begin{proof}
As in the proof of Theorem \ref{Eef-thm}, we obtain, using the energy balance (\ref{en-balance}),

\[
\frac{1}{T}\frac{1}{R_0^3}\int\limits_0^T \eta\,(f,u)\, dt = 
\varepsilon_0-\frac{1}{T}\frac{1}{R_0^3}\int\limits_0^T\eta_t\,\frac{\|u\|^2}{2}\,dt\ge\varepsilon_0-\frac{C_0}{T}e_0.
\]
Next, apply (\ref{e-bound}) to the last term of the inequality above:
\[
\frac{C_0}{T}e_0\le\frac{2\sqrt{2}C_0}{\pi^2}\frac{1}{T}\frac{R_0^2}{\nu}\frac{\|f\|}{R_0^{3/2}} e_0^{1/2}\le\frac{\sqrt{2}}{K}\frac{\|f\|}{R_0^{3/2}}e_0^{1/2}.
\]
Thus, using (\ref{ll1-a1}) to estimate $\varepsilon_0$, we obtain
\begin{equation}\label{ll1-p1}
\frac{1}{T}\frac{1}{R_0^3}\int\limits_0^T \eta\,(f,u)\, dt\ge \frac{2\sqrt{2}}{K} 
\frac{\|f\|}{R_0^{3/2}}e_0^{1/2}-\frac{\sqrt{2}}{K}\frac{\|f\|}{R_0^{3/2}}e_0^{1/2}=\frac{\sqrt{2}}{K}\frac{\|f\|}{R_0^{3/2}}e_0^{1/2}.
\end{equation}

\medskip

Recall that $\sum_i\psi_i\le K_2$, therefore
\begin{equation}\label{ll1-p2}
\begin{aligned}
\left|\ \,\iint\limits_{[0,T]\times\Omega} (\eta\sum_i\psi_i)\, f\cdot u \, dxdt\,\right| &\le \iint\limits_{[0,T]\times\Omega} K_2\eta\, |f| |u|\,dxdt \le K_2{T^{\frac{1}{2}}}\,\|f\|\, \left(\int\limits_0^T\|u\|^2\,\eta^{\delta}\, dt\right)^{\frac{1}{2}}\\& = TR_0^3\,K_2 \sqrt{2}\frac{\|f\|}{R_0^{3/2}}e_0^{1/2}.
\end{aligned}
\end{equation}
Now use (\ref{ll1-p1}) to obtain

\[
\left|\ \,\iint\limits_{[0,T]\times\Omega} (\eta\sum_i\psi_i)\, f\cdot u \, dxdt\,\right| \le K_2 K \left(TR_0^3\, \frac{\sqrt{2}}{K}\frac{\|f\|}{R_0^{3/2}}e_0^{1/2}\right)=KK_2\int\limits_0^T \eta\,(f,u)\, dt\;.
\]
 
\end{proof}

\medskip

We are now ready to estimate the $\langle (f,u)\rangle_R$-term.

\begin{lem}\label{ll2}
Under the assumptions on Lemma \ref{ll1}, for any $(K_1,K_2)$-average on scale $R$ we have:
\begin{equation}\label{ll2-conc}
\left|\langle (f,u) \rangle_R\right|=
\left|\frac{1}{T}\frac{1}{R^3}\frac{1}{n}\sum\limits_{i=1}^{n}\iint (\eta\psi_i)\, f\cdot u\, dxdt\right|\le
KK_2\,  \frac{\tau_{-1}}{\tau_f}\,  \frac{R_0^3}{R^3}\,\frac{1}{n}\, \frac{\|f\|}{R_0^{3/2}}e_0^{1/2},
\end{equation}

where 

\begin{equation}\label{ll2-n1}
\tau_f=\frac{\|f\|}{\|A^{\frac{1}{2}}f\|}
\quad\mbox{and}\quad 
\tau_{-1}=\left(\frac{\frac{1}{T}\frac{1}{R_0^3}\int\limits_0^T\eta^2\,\|A^{-\frac{1}{2}}u\|^2\, dt}{e_0}\right)^{\frac{1}{2}}
\end{equation}

\end{lem}

\begin{proof}
We start by noticing that 
\begin{align}
\label{ll2-p1}
&\left| \int\limits_0^T\eta\, (f,u)\, dt \right|\le
\left| \int\limits_0^T\eta\, (A^{1/2}f,A^{-1/2}u)\, dt \right|\le\int\limits_0^T\eta\, \|A^{1/2}f\|\,\|A^{-1/2}u\|\, dt \\
&\le T^{\frac{1}{2}} \|A^{\frac{1}{2}}f\|
\left( \int\limits_0^T\eta^2\,\|A^{-\frac{1}{2}}u\|^2\, dt \right)^{\frac{1}{2}}=T^{\frac{1}{2}}\, \frac{\tau_{-1}}{\tau_f} \|f\|\,\left( \int\limits_0^T\eta^{\delta}\,\frac{\|u\|^2}{2}\, dt \right)^{\frac{1}{2}}\nonumber.
\end{align}

Now, by Lemma \ref{ll1},

\[
\left|\langle (f,u) \rangle_R\right|=\left|\frac{1}{T}\frac{1}{R^3}\frac{1}{n}\iint (\eta\sum\limits_{i=1}^{n}\psi_i)\, f\cdot u\, dxdt\right|
\le \frac{1}{T}\frac{1}{R^3}\frac{1}{n}\,KK_2\int\limits_0^T\eta\,(f,u)\, dt,
\]
which, together with (\ref{ll2-p1}) implies (\ref{ll2-conc}).

\end{proof}

 Since $\tau_{-1}$ is the crucial scale in previous lemma, it is interesting to compare it to the Taylor scale $\tau_0$. 
 
\begin{lem}\label{lem2.0}
Assume $1/2<\delta<1$. Let 
\begin{equation}\label{better_tau2.0}
\tilde{\tau}_{-1}=\left(
\frac{\frac{1}{R_0^3}\frac{1}{T}\int\limits_0^T\eta^{2\delta-1}\|A^{-1/2}u\|^2}{e_0}
\right)^{\frac{1}{2}}
\end{equation}
then 
\begin{equation}\label{l2-conc}
\tilde{\tau}_{-1}\ge2\,\tau_0\;.
\end{equation}
\end{lem}

\begin{proof}

Note that we have:

\[2TR_0^3e_0=\int\limits_0^T \eta^{\delta}\|u\|^2=\int\limits_0^T (\eta^{\frac{1}{2}} A^{\frac{1}{2}}u, \eta^{\delta-\frac{1}{2}}A^{-\frac{1}{2}}u)\, dxdt\le
\left(\int\limits_0^T\eta\|A^{\frac{1}{2}}u\|^2\right)^{\frac{1}{2}}\left(\int\limits_0^T\eta^{2\delta-1}\|A^{-\frac{1}{2}}u\|^2\right)^{\frac{1}{2}}\]

and thus, (\ref{l2-conc}) follows (if we recall that $\tau_0=(\nu e_0/\varepsilon_0)^{1/2}$, and $\varepsilon_0$ is given by (\ref{periodic_global-enst})).

\end{proof}

{\rem\label{rem4.0}
{\rm Notice that $\tau_{-1}\le\tilde{\tau}_{-1}$. Thus, Lemma \ref{ll2} holds with $\tau_{-1}$ replaced by $\tilde{\tau}_{-1}$.
(Arguably, both $\tau_{-1}$ and $\tilde{\tau}_{-1}$ should in fact be of comparable size.) Also, (\ref{ll1-p1}) and (\ref{ll1-p2}) imply that the terms $\int\eta\,(f,u)\, dt$ and $\|f\|e_0^{1/2}$ have {\em the same scaling}: $\sim{\rm Gr}^{3/2}$. In particular this means that {\em the force, $f$, and the solution, $u$, are  aligned in $H$ over $[0,T]$}, since {\em the Cauchy-Schwarz inequality for $(f,u)$ is saturated}.
Therefore, the estimate $(\ref{ll2-conc})$ is sharp (as an order of ${\rm Gr}$), and, in order to control the force-related term in the energy balance, we will need to require that ``the bulk of oscillations'' of $f$ (characterized by the scale $\tau_f$) does not extend below the Taylor scale $\tau_0$.} }

\subsection{A better cascade theorem on $\Omega$}

Using the improved bound (\ref{ll2-conc}) instead of (\ref{force}), we obtain the following theorem.

\begin{thm}\label{better_periodic_casc-thm}
Let $C\in(0,2^{3/2})$ and $\alpha_0\in(0,1)$ be fixed. Assume $f\in D(A^{1/2})$ and $T>0$ be such that 

\begin{equation}\label{tt1-a1}
\|f\|\ge 2^{3}\, \frac{C_0^2K_2^2}{C^{4}\alpha_0^2}\,  \theta_f\, \frac{\nu^2}{R_0^{3/2}}\;
\end{equation}
and
\begin{equation}\label{tt1-a2}
T\ge \frac{2^{13/4} C_0}{\pi^2C^{3/2}}\frac{R_0^{2}}{\nu}.
\end{equation}

If $u\in\mathcal{A}_w$  is a suitable and Leray-Hopf solution that saturates Kolmogorov's dissipation law on the time interval $[0,T]$, i.e,
\begin{equation}\label{tt1-a3}
 \frac{C}{\theta_f}\,\frac{e_0^{3/2}}{R_0}\le\varepsilon_0\,(\le \frac{2\sqrt{2}}{\theta_f} \,\frac{e_0^{3/2}}{R_0})\;,
\end{equation}
and
\begin{equation}\label{tt1-a4}
\frac{2^{11/4}K_2}{\alpha_0C^{3/2}}\,\tau_{-1}\,\le\, \tau_f\;,
\end{equation}

then we will have energy cascade
\begin{equation}\label{tt1-conc1}
\frac{(1-\alpha_0)}{K_1}\, \varepsilon_0\,\le\, \langle \Phi\, \rangle_R \,\le\, K_2(1+\alpha_0) \,\varepsilon_0
\end{equation}
for any $(K_1,K_2)$-average on $\Omega$ for all $R$ inside the inertial range
\begin{equation}\label{tt1-conc2}
\left[
\beta_0\, \tau_{0},\; R_0
\right], \quad  \mbox{where}\quad\beta_0:=\left(\frac{2}{\alpha_0}C_0K_2\right)^{1/2}\;.
\end{equation}
Moreover, in this case
\begin{equation}\label{tt1-conc3}
\frac{C^{3/2}}{2^{15/4}} \,\theta_f^{1/2}\, \frac{\nu R_0^{5/4}}{\|f\|^{1/2}}\, \le\, \tau_0^2\,\le\, \frac{2^{3/2}}{C^2} \,\theta_f^{1/2}\, \frac{\nu R_0^{5/4}}{\|f\|^{1/2}}
\end{equation}

\end{thm}

\begin{proof}
First, note that by Theorem \ref{scaling}, the assumptions (\ref{tt1-a3}) and (\ref{tt1-a2}) imply (\ref{ll1-a1}) and (\ref{ll1-a2}) with 
\[K=\frac{1}{2^{3/4}C^{3/2}}.\] 
Therefore, the estimate from Lemma \ref{ll2} holds.

We proceed as in the proof of Theorem \ref{casc-thm}, but use (\ref{ll2-conc}) instead of (\ref{force}) to arrive to the following version of (\ref{ava_bounds}):

\begin{equation}\label{tt1_p1}
\begin{aligned}
&\frac{1}{n}\left(\frac{R_0}{R}\right)^3 \left[\varepsilon_0 - C_0K_2\frac{\nu}{R^2}\, e_0 - 2 K K_2\,\frac{\tau_{-1}}{\tau_f}\, \frac{\| f \|^2}{R_0^{3/2}}\, e_0^{1/2} \right]
\le \langle \Phi\, \rangle_R \le\\&\le
 K_2\, \frac{1}{n}\left(\frac{R_0}{R}\right)^3\left[
\varepsilon_0 + C_0\frac{\nu}{R^2}\, e_0 +  2 K \,\frac{\tau_{-1}}{\tau_f}\,\frac{\| f \|^2}{R_0^{3/2}}\, e_0^{1/2}\right]\:, 
\end{aligned}
\end{equation}

Now use (\ref{ll1-a1}) to bound $\frac{\| f \|^2}{R_0^{3/2}}\, e_0^{1/2}$, then apply  (\ref{R-n}) and factor out $\varepsilon_0$:

\[
\frac{1}{K_1}\, \varepsilon_0 \,\left( 1- C_0K_2\frac{\tau_0^2}{R^2}-2^{1/2}K^2K_2\frac{\tau_{-1}}{\tau_f}\right)
\le \langle \Phi\, \rangle_R \le
K_2 \varepsilon_0 \,\left( 1- C_0\frac{\tau_0^2}{R^2}-2^{1/2}K^2\frac{\tau_{-1}}{\tau_f}\right)
\]
Notice, that if 
\begin{equation}\label{tt1-p3}
C_0K_2\frac{\tau_0^2}{R_0^2}\le \frac{\alpha_0}{2} 
\end{equation}
and
\begin{equation}\label{tt1-p4}
2^{1/2}K^2K_2\frac{\tau_{-1}}{\tau_f}\le\frac{\alpha_0}{2}, 
\end{equation}
then (\ref{tt1-conc1}) hold for $R$ inside the interval given in (\ref{tt1-conc2}).

Obviously, (\ref{tt1-p4}) is equivalent to (\ref{tt1-a4}), as for (\ref{tt1-p3}), it indeed holds if the upper bound of (\ref{tt1-conc3}) and (\ref{tt1-a1}) hold.

It remains to prove (\ref{tt1-conc3}) holds. But it immediately follows from the bounds on $e_0$ and $\varepsilon_0$ from Remark \ref{rr2}.

\end{proof}

As was mentioned in Remark \ref{rem4.0}, the condition (\ref{tt1-a4}) is in fact a rather mild one. It states that the "force scale" $\tau_f$ is not too small, i.e. if $f$ acts mostly on the big scale. Thus, essentially, the above theorem establishes that Kolmogorov's dissipation law will trigger energy cascade in physical scales inside a periodic domain, provided the force is mild enough (i.e. the force scale $\tau_f$ dominates the scale
$\tau_{-1}$). Theorem \ref{scaling} (more precisely Remark \ref{rr2}) implies that, in terms of Grashof number (\ref{Gr-per}),  the width of the inertial range (in reciprocals of the length scales) is of order $\mathrm{Gr}^{1/4}$, thus the cascades are wider for higher magnitude forces. Finally, by the same theorem, the Reynolds number,  $\mathrm{Re}=e_0^{1/2}/(\nu R_0^{-1})$, grows like $\mathrm{Gr}^{1/2}$, and so our results prove that Kolmogorov's dissipation law and high Reynolds numbers imply wide energy cascades in the physical scales of periodic flows.

\medskip

We conclude with several remarks.

{\rem
{\rm
Comparing Theorems \ref{casc-thm} with \ref{better_periodic_casc-thm} we notice that in the periodic case, saturation of the Kolmogorov dissipation law (\ref{tt1-a3}) is equivalent to the saturation of the a priori bound (\ref{Eef0}) (see Remark \ref{rem4}.) Therefore (\ref{tt1-a3}) corresponds to the condition (\ref{ass1}) in Theorem \ref{casc-thm}. Thus, philosophically, Theorem \ref{casc-thm} states that in generic settings, saturation of Kolmogorov dissipation law plus small Taylor scale imply energy cascade. Theorem \ref{better_periodic_casc-thm} states that in the case of the whole domain in periodic case, if force is big enough, Kolmogorov law in fact {\em implies} that Taylor scale is small and is alone responsible for the cascade.
}}

{\rem
{\rm Overall the results of Section \ref{periodic_casc}, obtained for the physical scales, mirror  closely the ones obtained in \cite{DFJ09} in the framework of Fourier scales and statistical solutions (see Remark \ref{rem04}). Recall, \cite{DFJ09} shows that Kolmogorov's dissipation law (defined there in terms of generalized limits) induces {\em the same} restrictive scaling on energy and its dissipation rate as the one obtained here. Recall that, as noted in Remark \ref{rem04}, in the context of Fourier scales, as $\mathrm{Gr}\to\infty$, this scaling combined with the cascade condition from \cite{FMRT01a}, {\em automatically} implies energy cascade up to the Taylor scale.
}}

\medskip

{\rem{\rm We should stress that Kolmogorov's dissipation law (\ref{tt1-a3}) is responsible for the peculiar alignment of $f$ and $u$, which makes the estimate (\ref{ll2-conc}) possible. In general case, in absence of additional information on alignment of $f$ and $u$, it is impossible to a priori eliminate the case when the Cauchy-Schwarz inequality $(f,u)\le\|f\|\,\|u\|$ saturated on the interval $[0,T]$ even if one assumes $\tau_0\ll \|f\|/\|A^{1/2}f\|$ (which would mean that $f$ is geometrically very different from $u$.)

\medskip

For example, one can construct vectors in $H$: $u=v_1+\alpha v_n$ and $f=v_1$, where $\{v_i\}$ is an orthonormal basis of eigenvectors of $A$, arranged in the increasing order of eigenvalues. If $\alpha=\lambda_n^{-1/4}$ ($\lambda_n$ is an eigenvalue corresponding to $v_n$) an $n\gg 1$, then we will in fact have $(f,u)\approx \|f\|\,\|u\|$, while $"\tau_0"=\|u\|/\|A^{1/2}u\|\ll\|f\|/\|A^{1/2}f\|$. If something similar happens for an actual solution of the NSE on $[0,T]$, then (\ref{force}) is the best estimate available, which, as we noticed in Remark \ref{rem1}, will not be useful in establishing cascade on the whole $\Omega$. 
}
}

\medskip

{\rem
{\rm
Notice that because of periodicity, the global flux on $\Omega$ is zero: 
\[\Phi_{\Omega}=\frac{1}{T}\frac{1}{R_0^3}\int\limits_0^{T} \eta\,(B(u,u),u)\, dt=0\;.\]
Therefore, the flux density $(u\cdot\nabla u) u\cdot u +\nabla p$ {\em must be a sign-varying quantity}. The fact that {\em all} the $(K_1,K_2)$-averages  across a wide range of scales are comparable to $\varepsilon_0$ is therefore truly significant.
}}


\section[A]{Appendix - the link between suitable and Leray-Hopf solutions in the periodic case.}\label{appx}

As mentioned above, in the space-periodic settings the suitable weak solutions of the NSE will in fact also be the Leray-Hopf solutions with a slightly modified global energy inequality (the inequality is localized in time, in contrast to the classical Leray-Hopf case). For completeness, we present the poof of this fact. It is worth noticing, that the converse statement is an open problem, linked with the general regularity problem of the NSE. 

\begin{thm}\label{siut->leray}
Let $u,p$ be a suitable solution on $[0,T]$ to the space-periodic  NSE. Then $u$ will satisfy the NSE in $V'$ (i.e. $u$ is a weak solution -- see (\ref{weak_form})) and for any time cut-off $\eta\in C^1_0([0,T])$ the global energy inequality will hold:
\begin{equation}\label{en_ineq1}
\nu \int_0^T\eta\, \|\nabla u\|^2\le \int_{0}^T\eta_t \frac{\|u\|^2}{2} +\int_0^T\eta\,(f,u)\;.
\end{equation}
\end{thm}

\begin{proof}
Let $L>0$ be the period, and $\Omega=[0,L]^3$ be a fixed box of size $L$. We start with the following claim.

\begin{claim}
There exists a finite family of functions $\{\psi_j\}$ such that:
\begin{enumerate}
\item Each $\psi_j\in C^{\infty}(\mathbb{R}^n)$, $\psi_j$ are $L$-periodic, $1\ge\psi_j\ge0$.
\item For each $\psi_j$ there exists a size $L$ box $\Omega_j\subset\mathbb{R}^n$, such that the restriction of $\psi_j$ on the {\em interior} 
of $\Omega_j$ has a compact support.
\item $\sum \psi_j = 1$ on $\mathbb{R}^n$.
\end{enumerate}
\end{claim}

\begin{proof} (of the Claim) 
Let $f_1\in C^{\infty}(\mathbb{R})$ be an $L$-periodic function, $0\le f_1 \le1$, such that $f_1=1$ on $[(1/3)L,(2/3)L]$  and $f_1=0$ on $[0,(1/6)L]\cup[(5/6)L,1]$. Now $f_2=1-f_1$ will also be a $C^{\infty}$, nonnegative, $L$-periodic function restriction of which on the interior of (the 1-dimensional periodic box) $[-L/2,L/2]$  is compactly supported. Moreover, $f_1+f_2=1$ on $\mathbb{R}$.

In $\mathbb{R}^n$ we have 
\[1=\prod\limits_{k=1}^n(f_1(x_k)+f_2(x_k))=\sum\limits_{i_1,\dots,i_n=1}^2 f_{i_1}(x_1)\cdot\ldots\cdot f_{i_n}(x_n)=
\sum\limits_{i_1,\dots,i_n=1}^2 \psi_{i_1,\dots,i_n}(x)\,
,\]
where $\psi_{i_1,\dots,i_n}(x_1,\ldots,x_n)=f_{i_1}(x_1)\cdot\ldots\cdot f_{i_n}(x_n)$ satisfy the the conditions (1)-(3) (after a re-indexing).
\end{proof}

Returning to the proof of the theorem, let $\psi_j$ be the test functions satisfying (1)-(3) in Claim 1. Then on for all $j$ the local energy inequalities are satisfied:
\begin{equation}\label{en_ineq2}
\int\limits_0^T\int\limits_{\Omega_j}\left(\frac{|u|^2}{2}+p\right)\, u\cdot\nabla\phi_j\, \ge\,
\nu \int\limits_0^T\int\limits_{\Omega_j} |\nabla u|^2\phi_j\, -\, \int\limits_0^T\int\limits_{\Omega_j}\frac{|u|^2}{2}(\partial_t\phi_j-\nu\Delta\phi_ij) \,-\, \int\limits_0^T\int\limits_{\Omega_j} (f,u)\phi_j\;,
\end{equation}
where $\phi_j=\eta\psi_j$ and $\Omega_j$ from the Claim 1, (2). But since all the functions involved are $L$-periodic in $x$, the integrals over $\Omega_j$ equal to the integrals over $\Omega$ a.e. in $t$. Thus, writing $\Omega$ instead of $\Omega_j$ in (\ref{en_ineq2}) and summing up in $j$ (keeping in mind $\sum \phi_j=\sum \eta\psi_j=\eta$), we obtain (\ref{en_ineq1}).

It remains to show $u$ solves the NSE in $V'$. But since $u$ is a suitable weak solution, then $u\in L^2([0,T],V)$ and solves the NSE in the sense of distributions:
\begin{equation}\label{distrib_nse}
-\int\limits_0^T\int\limits_{\mathbb{R}^3} u\cdot \phi_t - \int\limits_0^T\int\limits_{\mathbb{R}^3} (\nabla u:\nabla\phi)-\int\limits_0^T\int\limits_{\mathbb{R}^3} (u\cdot\nabla)\phi\cdot u= \int\limits_0^T\int\limits_{\mathbb{R}^3} f\cdot\phi\,,
\end{equation}
for any divergence-free $\phi\in C_0^{\infty}(\mathbb{R}\times \mathbb{R}^3, \mathbb{R}^3)$. Note that since $u$, $p$, and $f$ are space periodic, then by extending $C_0^{\infty}$ test functions with the support inside an $L$-periodic box to the whole $\mathbb{R}^3$ (and considering sums of such functions, like in the proof of Claim 1), we can see that (\ref{distrib_nse}) also holds for $\phi$ that are $C^{\infty}$ and $L$-periodic in the space variable, integrated over $\Omega$.

But the density properties of $C^{\infty}$ in $H^1$ imply that
for any $v\in V$ there exists a sequence of divergence free space periodic $C^{\infty}$-functions,  $\{\psi_k\}$ such that $\psi_k\to v$ in $H^1(\Omega,\mathbb{R}^3)$. Then for $\phi_k=\eta\psi_k$, distributional pairing above implies, as $k\to\infty$, the duality pairing of the NSE with $v$ in $V'$:
\begin{equation}\label{distr_nse}
-\int\limits_0^T\int\limits_{\Omega} \eta_t\, u\cdot v\,- \int\limits_0^T\int\limits_{\Omega} \eta\,(\nabla u:\nabla v)-\int\limits_0^T\int\limits_{\Omega} \eta\,(u\cdot\nabla)v\cdot u= \int\limits_0^T\int\limits_{\Omega} \eta\,f\cdot v\,.
\end{equation}
This means that for any $v\in V$, in the sense of distributions,
\[F_t=G,\]
where, using the notation from Subsection \ref{per-sec},  $F=(u,v)$, $G=(B(u,v),u)-(A^{1/2}u,A^{1/2}v)+(f,v)$. Thus,
(\ref{weak_form}) is satisfied,
i.e, the suitable weak solution $u$ satisfies the NSE in $V'$.

\end{proof}

\begin{rem}{\rm
Notice that by (\ref{distr_nse}), both $F$ and $G$ are locally $L^1$ in time ($F$, $G$ are form the proof of Theorem \ref{siut->leray}.) Therefore, $F$ is absolutely continuous and $F_t=G$ a.e., in other words, (\ref{weak_form}) {\em holds  a.e. in} $t$.
}\end{rem}

\medskip

The next result is an elementary observation on adapting the Leray-Hopf energy inequality to the localized in time case. 

\begin{thm}\label{str_en_ineq}
Assume $u$ is the Leray-Hopf solution of the space-periodic NSE, and let $\eta\in C^{2}([t_0,t_1],\mathbb{R}_+)$, where $t_0$ is the Lebesgue point of $\|u(\cdot)\|^2$, and $t_1>t_0$. Then,
\begin{equation}\label{LH-gen}
\eta(t_1)\frac{\|u(t_1)\|^2}{2}-\eta(t_0)\frac{\|u(t_0)\|^2}{2} - \int\limits_{t_0}^{t_1}\eta_t\frac{\|u\|^2}{2}\le - \nu\int\limits_{t_0}^{t_1} \eta\,\|\nabla u\|^2 + \int\limits_{t_0}^{t_1} \eta\, (f,u)\;.
\end{equation}

\end{thm}

\begin{proof}

Recall that the Leray-Hopf solutions satisfy the energy inequality (\ref{leray_en_ineq}) for all $t_0\ge0$ - Lebesgue points of $\|u(\cdot)\|^2$ and for all $t\ge t_0$.
Moreover, the set of Lebesgue points is dense in $[0,\infty)$. 
Additionally, $u(t)$ is regular on a countable union of open intervals of the full measure inside any $[0,T]$, which means $\|u(t)\|^2$ is continuous a.e. Finally, $\|u(t)\|^2$ is bounded on $[0,\infty)$ -- a direct consequence of the Leray-Hopf energy inequality (\ref{leray_en_ineq}). Therefore, $\|u(t)\|^2$ is Riemann integrable on any compact interval in $[0,\infty)$.

\medskip

Now let $t_0\ge0$ be a Lebesgue point for $\|u(t)\|^2$, $t_1>t_0$, and let $\eta\in C^2([t_0,t_1],\mathbb{R}_+)$. For a $\delta>0$, consider a Riemann partition 
$\{s_j\}_{j=0,n_{\delta}}$, where each $s_j$ is a Lebesgue point for $\|u(t)\|^2$, $t_0=s_0<s_1<\dots<s_{n_{\delta}}=t_1$, and each $s_{j}-s_{j-1}<\delta$.  If for $j<n_{\delta}$ we define the step functions 
\[
\eta_j(t)=\left\{\begin{array}{ll}
\eta(s_j),\quad &t\in[s_j,s_{j+1});\\
0, &\mbox{otherwise},
\end{array}\right.
\]
then, as $\delta\to 0$,  $\sum\eta_j$ converges to $\eta$ a.e. on $[t_0,t_1]$. Also, for each $j=0,\dots,n_{\delta}-1$, the Leray-Hopf energy inequality (\ref{leray_en_ineq}) on $[s_j,s_{j+1}]$ implies:
\[
\eta(s_{j})\frac{\|u(s_{j+1})\|^2}{2}-\eta(s_j)\frac{\|u(s_j)\|^2}{2} \le - \nu\int\limits_{t_0}^{t_1} \eta_j\,\|\nabla u\|^2 + \int\limits_{t_0}^{t_1} \eta_j\, (f,u)\;.
\]
Summing in $j$, we obtain
\begin{equation}\label{in1}
\sum\limits_{j=0}^{n_{\delta}-1}\eta(s_{j})\frac{\|u(s_{j+1})\|^2}{2}-\sum\limits_{j=0}^{n_{\delta}-1}\eta(s_j)\frac{\|u(s_j)\|^2}{2} \le - \nu\int\limits_{t_0}^{t_1} \sum\limits_{j=0}^{n_{\delta}-1}\eta_j\,\|\nabla u\|^2 + \int\limits_{t_0}^{t_1} \sum\limits_{j=0}^{n_{\delta}-1}\eta_j\, (f,u)\;.
\end{equation}
Notice that by the Lebesgue dominated convergence theorem, as $\delta\to0$, $\int_{t_0}^{t_1} \sum\eta_j\,\|\nabla u\|^2\to\int_{t_0}^{t_1} \eta\,\|\nabla u\|^2$, while 
$\int_{t_0}^{t_1} \sum\eta_j\, (f,u)\to\int_{t_0}^{t_1} \sum\eta\, (f,u)$.

\medskip

To deal with the right-hand side of (\ref{in1}), write
\[
\begin{aligned}
&\sum\limits_{j=0}^{n_{\delta}-1}\eta(s_{j})\frac{\|u(s_{j+1})\|^2}{2}-\sum\limits_{j=0}^{n_{\delta}-1}\eta(s_j)\frac{\|u(s_j)\|^2}{2}=\\
&=\eta(s_{n_{\delta}-1})\frac{\|u(s_{n_{\delta}})\|^2}{2}-\eta(s_{0})\frac{\|u(s_{0})\|^2}{2}-
\sum\limits_{j=1}^{n_{\delta}-1}(\eta(s_{j})-\eta(s_{j-1}))\frac{\|u(s_{j})\|^2}{2}\,.
\end{aligned}
\]
Notice that
\[
\sum\limits_{j=1}^{n_{\delta}-1}(\eta(s_{j})-\eta(s_{j-1}))\frac{\|u(s_{j})\|^2}{2}=\sum\limits_{j=1}^{n_{\delta}-1}\eta'(s_j)\frac{\|u(s_{j})\|^2}{2}(s_j-s_{j-1}) -
\sum\limits_{j=1}^{n_{\delta}-1}\eta''(\xi_j)\frac{\|u(s_{j})\|^2}{2}(s_j-s_{j-1})^2\,,
\]
where $\xi_j\in [s_{j-1},s_j]$. We can identify the first sum in the right-hand side as the Riemann sum (with a missing 0-th term) corresponding to $\int_{t_0}^{t_1}\eta_t\frac{\|u\|^2}{2}$, while the absolute value is the second sum is bounded above by:
\[\left|  \sum\limits_{j=1}^{n_{\delta}-1}\eta''(\xi_j)\frac{\|u(s_{j})\|^2}{2}(s_j-s_{j-1})^2 \right|\le M\delta R_{n_{\delta}}\,,\]
where $M=\sup\{\eta''(t):\ t\in[t_0,t_1]\}<\infty$, and $R_{n_{\delta}}$ is the Riemann sum  corresponding to $\int_{t_0}^{t_1}\frac{\|u\|^2}{2}<\infty$. Thus, this term converges to zero as $\delta\to 0$.

Consequently, (\ref{LH-gen}) is obtained form (\ref{in1}) by letting $\delta\to 0$.

\end{proof}

\begin{rem}{\rm
Clearly, (\ref{LH-gen}) implies the usual Leray-Hopf inequality (\ref{leray_en_ineq}).
Also, (\ref{LH-gen}) implies (\ref{en_ineq1}) for $\eta\in C_0^2((0,T),\mathbb{R}_{+})$, and therefore, any Leray-Hopf solution will automatically satisfy the localized in time global energy inequality considered in Section \ref{en-enst-sec}. 

}\end{rem}


\bibliographystyle{plain}
\bibliography{3d_forced_turbulence-1}

\end{document}